\documentclass[12pt]{article}%
\usepackage{amssymb,amsfonts}
\usepackage{amscd}
\usepackage{amsmath}
\usepackage{amsfonts}
\usepackage{amssymb}
\providecommand{\U}[1]{\protect\rule{.1in}{.1in}}
\providecommand{\U}[1]{\protect\rule{.1in}{.1in}}
\newtheorem{theorem}{Theorem}[section]
\newtheorem{lemma}[theorem]{Lemma}
\newtheorem{proposition}[theorem]{Proposition}

\newtheorem{example}[theorem]{Example}
\newtheorem{definition}[theorem]{Definition}

\newenvironment{proof}{\bf Proof. \rm}{$\Box$}
\newcommand{\be}{\begin{equation}}
\newcommand{\ee}{\end{equation}}

\newcommand{\cD}{\mathcal{D}}

\newcommand{\cM}{\mathcal{M}}
\newcommand{\cF}{\mathcal{F}}
\newcommand{\cL}{\mathcal{L}}

\newcommand{\cR}{\mathcal{R}}
\newcommand{\cS}{\mathcal{S}}
\newcommand{\cH}{\mathcal{H}}
\newcommand{\bC}{\mathbb{C}}

\makeatother

\begin{document}

\title{Invariant subspaces for certain tuples of operators with applications to reproducing kernel correspondences}


\author{ Baruch Solel
\\Department of Mathematics, Technion\\32000 Haifa, Israel
\\e-mail: mabaruch@technion.ac.il}


\date{}

\maketitle

\begin{abstract}
The techniques developed by Popescu, Muhly-Solel and Good for the study of algebras generated by weighted shifts are applied to generalize results of Sarkar and of  Bhattacharjee-Eschmeier-Keshari-Sarkar concerning dilations and invariant subspaces for commuting tuples of operators. In that paper the authors prove Beurling-Lax-Halmos type results for commuting tuples $T=(T_1,\ldots,T_d)$ operators that are contractive and pure; that is $\Phi_T(I)\leq I$ and $\Phi_T^n(I)\searrow 0$ where $$\Phi_T(a)=\Sigma_i T_iaT_i^*.$$

Here we generalize some of their results to commuting tuples $T$ satisfying similar conditions but for $$\Phi_T(a)=\Sigma_{\alpha \in \mathbb{F}^+_d} x_{|\alpha|}T_{\alpha}aT_{\alpha}^*$$ where $\{x_k\}$ is a sequence of non negative numbers satisfying some natural conditions (where $T_{\alpha}=T_{\alpha(1)}\cdots T_{\alpha(k)}$ for $k=|\alpha|$). In fact, we deal with a more general situation where each $x_k$ is replaced by a $d^k\times d^k$ matrix.

We also apply these results to subspaces of certain reproducing kernel correspondences $E_K$ (associated with maps-valued kernels $K$) that are invariant under the multipliers given by the coordinate functions.
\end{abstract}
\maketitle

\section{Introduction}
A famous theorem due to Beurling, Lax and Halmos (see \cite[Corollary 3.26]{RR}) describes the shift-invariant subspaces of $H^2_{G}(\mathbb{D})$ (where $H^2(\mathbb{D})$ is the classical Hardy space, $G$ is a Hilbert space and $H^2_G(\mathbb{D})$ is the vector-valued Hardy space and is isomorphic to $H^2(\mathbb{D})\otimes G$). The theorem states that every shift invariant subspace of $H^2_G(\mathbb{D})$ is given by a partially isometric image of a vector-valued Hardy space $H^2_{D}(\mathbb{D})$. In fact, this partial isometry can be given by an inner function on $\mathbb{D}$ with values in $B(D,G)$.

This result was proved to be very important and there are many generalizations of it to various contexts. I will mention here only the ones that are most relevant to the discussion here.

The Hardy space $H^2(\mathbb{D})$ is a reproducing kernel Hilbert space (RKHS) where the kernel is the Szeg\"{o} kernel $K(z,w)=\frac{1}{1-z\overline{w}}=\Sigma_{k=0}^{\infty}z^k\overline{w^k}$ ($z,w\in \mathbb{D}$) and it has the Nevanlinna-Pick interpolation property. In \cite{McT2000} McCullough and Trent extended the BLH theorem to a general reproducing kernel Hilbert space $H_K$ where $K$ has the Nevanlinna-Pick property, in place of the Hardy space.
For the case of weighted Bergman spaces, Ball and Bolotnikov (\cite{BB2013a}, \cite{BB2013b}) studied the invariant subspaces and showed that they are the image (by  partially isometric multipliers) of vector valued Hardy spaces.


In \cite{S2015} J. Sarkar proved a BLH-type theorem for subspaces that are invatiant for a pure contraction and applied it to shift-invariant subspaces of RKHSs with special property that he calls ``analytic" and that ensures that the coordinate function is a contractive multiplier. The Hardy space and the Bergman space are examples of analytic RKHSs. It is proved in these cases that an invariant subspace is a partially isometric image of a vector-valued Hardy space. Since the weighted Bergman spaces are analytic RKHSs, these results extend the results of Ball-Bolotnikov.


In \cite{S2016} J. Sarkar extends the results of \cite{S2015} to invariant subspaces of a pure row contraction of commuting operators. The place of the Hardy space is now played by the Drury-Arveson space (i.e, the symmetric Fock space). In \cite{BEKS} the authors continued the study of such commuting tuples and discussed also uniqueness (of dilations and of the partially isometric multipliers), the wandering subspaces associated with invariant subspaces and $K$-inner functions.

Another BLH type theorem that is relevant to our analysis here is in Popescu's \cite{Po2010}. In particular, \cite[Theorem 3.3]{Po2010}, where he characterized the invariant subspaces under the constrained weighted shifts associated with a noncommutative variety is closely related to Theorem~\ref{BLHC} (3) below. One should note though that his varieties are more general than what we use here but his definition of weights is more restrictive (the matrices $\{X_k\}$, $\{R_k\}$ and $\{Z_k\}$ that we define below are, in his analysis, assumed all to be diagonal).

A general BLH theorem for RKHS $H_K$ with the property that $K$ has a complete Nevanlinna-Pick factor $s$ (that is, $K/s$ is a positive kernel) was recently proved by Clou\^{a}tre, Hartz and Schillo (\cite{CHS}). They proved that, for a Hilbert space $\mathcal{E}$, a subspace $\mathcal{S}\subseteq H_K\otimes \mathcal{E}$ is invariant under the multipliers of $H_s$ (which, under the condition of the positivity of $K/s$, are also multipliers of $H_K$) if and only if there is a Hilbert space $\cF$ such that $\cS$ is the image of $H_s\otimes \cF$ by a partially isometric multiplier.

In this paper, we study invariant subspaces for a tuple of commuting operators satisfying certain inequality and are pure in some natural sense. In \cite{S2016} Sarkar studied  commuting tuples $T=(T_1,\ldots,T_d)$ that are contractive and pure; that is $\Phi_T(I):=\Sigma_{i=1}^{d} T_iT_i^*\leq I$ and $\Phi_T^n(I)\searrow 0$ where $$\Phi_T(a)=\Sigma_i T_iaT_i^*.$$  Here, we replace this condition by a much more general one. Namely we write
$$\Phi_T(a)=\Sigma_{\alpha \in \mathbb{F}^+_d} x_{|\alpha|}T_{\alpha}aT_{\alpha}^*$$ where $\{x_k\}$ is a sequence of non negative numbers satisfying some natural conditions (where $T_{\alpha}=T_{\alpha(1)}\cdots T_{\alpha(k)}$ for $k=|\alpha|$). Thus, the tuples we study are commuting and satisfy $$\Sigma_{\alpha \in \mathbb{F}^+_d} x_{|\alpha|}T_{\alpha}T_{\alpha}^*\leq I
 $$ and $\Phi_T^n(I)\searrow 0$.
 In fact, we deal with a more general situation where each $x_k$ is replaced by a $d^k\times d^k$ matrix $X_k$ .

 In Theorem~\ref{Tinvariant} we present the BLH theorem for such tuples. The place of $H^2(\mathcal{D})$ (in \cite{S2015}) or the Drury-Arveson space (in \cite{S2016}) is now played by the ``weighted Drury-Arveson space" $\cF(\cR)$ that is defined and studied in Section 3. Among other things, it is shown there that it is the RKHS for some kernel defined on a domain that we denote by $D(X,\bC)$ (defined in (\ref{Dc})). Both the kernel and the domain are defined in terms of the sequence $\{X_k\}$. The motivation for these definitions comes from the study of weighted Hardy algebras in \cite{MSWS}.

 The reader does not have to be familiar with the analysis of the weighted Hardy algebras and their representations as studied in \cite{MSWS} and in \cite{GThesis} but the results of these papers provide some of the motivation for our analysis here.

 In Section 5 we discuss reproducing kernel correspondences $E_K$ associated with a map-valued kernel $K$ and we apply Theorem~\ref{Tinvariant} to the row of multipliers by the coordinate functions. This is done in Theorem~\ref{BLH} and Theorem~\ref{BLHC}. In the study of the reproducing kernel correspondences we use results of \cite{MThesis} and of \cite{GThesis}. Since, unfortunately, these results  have not been published yet, we shall provide all the details necessary for our discussion here.

 In order to apply Theorem~\ref{Tinvariant} to the multipliers of the reproducing kernel correspondence $E_K$ given by the coordinate functions $M_{S_i}$, we need to show that $M_S:=(M_{S_1},\ldots,M_{S_d})$ satisfies the conditions $\Phi_{M_S}(I)\leq I$ and $\Phi^n_{M_{S}}(I)\searrow 0$. It turns out (Theorem~\ref{condK}) that this holds if and only if the kernel $K^R_c$ (defined at the beginning of Section 5) is a factor of $K$. Since $K^R_c$ was shown by Good (in \cite{GThesis}) to have the complete Nevanlinna-Pick property, this fits well with the results of Clou\^{a}tre-Hartz-Schillo (\cite{CHS}) mentioned above.

 Since we will consider here both correspondences (which are $C^*$-Hilbert modules) and Hilbert spaces and it is customary, in the theory of $C^*$-Hilbert modules, to have inner products that are linear in the \emph{second} term, we will use this convention also for Hilbert spaces. In particular, when studying a reproducing kernel Hilbert space $H_K$ with kernel functions $k_z$, we shall write $\langle k_z,k_w \rangle=K(z,w)$.

 Finally, I wish to acknowledge helpful discussions with J. Eschmeier and with J. Sarkar regarding the results of this paper.

\section{Preliminaries}
We fix $0<d<\infty$. The following definition was introduced in \cite{MSWS} for the more general case where the Hilbert space $\bC^d$ was replaced by a $W^*$-correspondence.
\begin{definition}\label{admissible}
A sequence $\{X_k\}_{k=1}^{\infty}$ of operators will be called admissible if it satisfies the following conditions:
\begin{enumerate}
\item[1.] $X_k\in B(\bC^d \otimes \bC^d \otimes \cdots \otimes \bC^d)=B((\bC^d)^{\otimes k})$.
\item[2.] $X_1$ is invertible and $X_k\geq 0$ for all $k$.
\item[3.] $\limsup ||X_k||^{\frac{1}{k}}<\infty$.

\end{enumerate}
\end{definition}

Write $\mathbb{F}_d^+$ for the words $\alpha=\alpha(1)\alpha(2) \ldots \alpha(k)$ on $d$ generators (written $\{1,\ldots,d\}$). For such a word $|\alpha|=k$. Note that $X\in B((\bC^d)^{\otimes k})$ can be written as a matrix $(x_{\alpha,\beta})$ indexed by words $\alpha, \beta\in \mathbb{F}_d^+$ of length $k$ and a sequence $\{X_k\}$ with $X_k\in B((\bC^d)^{\otimes k})$ can be written as a matrix $(x_{\alpha,\beta})_{\alpha,\beta \in \mathbb{F}_d^+}$.

Associated with an admissible sequence $\{X_k\}$ we have another sequence of operators denoted $\{R_k\}_{k=0}^{\infty}$ where $R_0=I\in B(\bC)$ and, for $k\geq 1$,
\be\label{Rk}
R_k=(\Sigma_{l=1}^k \Sigma_{\alpha\in F(k,l)} \otimes_{i=1}^l X_{\alpha(i)})^{\frac{1}{2}}
\ee
where $$F(k,l)=\{\alpha:\{1,\ldots,l\}\rightarrow \mathbb{N} :\;\Sigma_{i=1}^l \alpha(i)=k \}.$$
Note that each $R_k$ is positive and invertible.
Here, also, we can write $\{R_k^2\}_{k=0}^{\infty}$ as a matrix $(r^2_{\alpha,\beta})_{\alpha,\beta \in \mathbb{F}_d^+\cup \{\emptyset\},|\alpha|=|\beta|}$.

The relationship between the $X_k$'s and the $R_k$'s is also given by
$$R_k^2= \Sigma_{l=1}^k X_l\otimes R_{k-l}^2, \;\; k\geq 1 $$ by \cite[(4.7)]{MSWS}.
If $d=1$ (so that $X_k=x_k$ and $R_k=r_k$ are scalars) then we have
\be\label{xr}
\frac{1}{1-\Sigma_{k\geq 1} x_k z^k}=\Sigma_{k\geq 0}r_k^2z^k
\ee  and
\be\label{rx}
1-\frac{1}{\Sigma_{k\geq 0}r_k^2z^k }=\Sigma_{k\geq 1} x_k z^k.
\ee

Given a tuple $T=(T_1,\ldots,T_d)$ of operators in $B(H)$, it will be convenient to view $T$ as an operator from $H^{(d)}$ to $H$ and
 for $k\geq 1$, we write $T^{(k)}:= (T_{\alpha})_{\alpha\in \mathbb{F}_d^+, |\alpha|=k}$ (viewed as a row of operators or, equivalently, as a bounded operator from $(\bC^d)^{\otimes k}\otimes H$ to $H$.) One can also write $$T^{(k)}= T(I_E\otimes T)(I_{E^{\otimes 2}}\otimes T)\cdots (I_{E^{\otimes (k-1)}}\otimes T): E^{\otimes k}\otimes H \rightarrow H$$ where $E=\bC^d$ (as $T:\bC^d \otimes H\rightarrow H$).

Given a tuple $T=(T_1,\ldots,T_d)$ and an admissible sequence $X=\{X_k\}=(x_{\alpha,\beta})$ such that the sum $\Sigma_{k=1}^{\infty}T^{(k)}(X_k\otimes I_H)T^{(k)*}=\Sigma_{|\alpha|=|\beta|} T_{\alpha}x_{\alpha,\beta}T_{\beta}^*$ is $w^*$-convergent,

we write $\Phi_T:B(H)\rightarrow B(H)$ for the (completely positive) map
\be\label{PhiT}
\Phi_T(a)=\Sigma_{k=1}^{\infty}T^{(k)}(X_k\otimes a)T^{(k)*}=\Sigma_{|\alpha|=|\beta|} T_{\alpha}x_{\alpha,\beta}aT_{\beta}^*
\ee where $T_{\alpha}=T_{\alpha(1)}T_{\alpha(2)}\cdots T_{\alpha(k)}$ if $|\alpha|=k$.

Also, we write
\be
D(X,H)=\{T=(T_1,\ldots,T_d):\; T_i\in B(H),\; ||\Phi_T(I)||<1 \}
\ee
and
\be\label{Dc}
D_c(X,H)=\{T\in D(X,H):\; T_iT_j=T_jT_i ,\; i,j \in \{1,\ldots,d\}\}
\ee  (the commuting $d$-tuples in $D(X,H)$),

$$\overline{D}_c(X,H)=\{T\in \overline{D}(X,H):\; T_iT_j=T_jT_i ,\; i,j \in \{1,\ldots,d\}\}$$ and
$$\overline{D}_{cp}=\{T\in \overline{D}(X,H):\; T_iT_j=T_jT_i ,\; i,j \in \{1,\ldots,d\},\; \Phi_T^m(I)\searrow 0\;\}$$ where a tuple satisfying $\Phi_T^m(I)\searrow 0$ is said to be pure.

\begin{example}
	Suppose each $X_k$ is a scalar. That is, $X_k=x_kI_{(\mathbb{C}^d)^{\otimes k}}$. Then $\Phi_T(a)=\Sigma_{|\alpha|=|\beta|}x_{|\alpha|}T_{\alpha}aT_{\alpha}^*$ and $$D(X,H)=\{T=(T_1,\ldots,T_d):\; T_i\in B(H),\;\Sigma_{|\alpha|=|\beta|}x_{|\alpha|}T_{\alpha}T_{\alpha}^*<1\;\}.$$
	
	In particular, if $x_1=1$ and $x_k=0$ for $k>1$, we get $\Phi_T(a)=\Sigma_{i=1}^d T_iaT_i^*$ and $D(X,H)=\{T=(T_1,\ldots,T_d):\; T_i\in B(H),\; \Sigma_i T_iT_i^*<1 \}$
\end{example}

\begin{example}
In the following examples we fix $H=\mathbb{C}$ , so we write $z=(z_1,\ldots,z_d)$ in place of $T$ and $\Phi_z(1)=\Sigma_{|\alpha|=|\beta|} z_{\alpha}x_{\alpha,\beta}\overline{z_{\beta}}$.
\begin{enumerate}
\item[(1)] When $d=1$, $\Phi_z(1)=\Sigma x_k|z|^{2k}$. Using (\ref{rx}), we see that $\Phi_z(1)<1$ if and only if the series $\Sigma_{k\geq 0}r_k^2 |z|^{2k}$ converges. Since this is a power series, the set $D(X,\mathbb{C})$ is a disc in $\bC$ in this case.
\item[(2)] Now set $d=2$, $X_1$ is the diagonal matrix $diag(1,2)$ and $X_k=0$ for $k>1$. Then $\Phi_z(1)=|z_1|^2+2|z_2|^2$ and $D(X,\mathbb{C})$ is a convex set whose restriction to $\mathbb{R}^2$ is an ellipse.
\item[(3)] Set $d=2$, $X_1=I_2$, $X_2$ is the diagonal matrix $diag(1,120,120,1)$ and $X_k=0$ for $k>2$. Then $\Phi_z(1)=|z_1|^2+|z_2|^2+|z_1|^4+|z_2|^4+240|z_1z_2|^2$. In this case $(0,1/2), (1/2,0) \in D(X,\bC)$ but $(1/4,1/4)$ is not in $D(X,\bC)$. Thus $D(X,\bC)$ is not convex.

\end{enumerate}

\end{example}
In what follows it will be convenient to keep writing $E$ for $\bC^d$ . We shall also write $\cF(E)$ for the full Fock space $\cF(E)=\cF(\bC^d)=\Sigma_{k=0}^{\infty}(\bC^d)^{\otimes k}$.

The following discussion and lemma will be useful to determine whether a commuting tuple is pure.

For a tuple $T=(T_1,\ldots,T_d)$ of commuting operators in $B(H)$ we write $b(T)$ for the row
\be\label{bT}
b(T)=(T^{(k)}(X_k^{1/2}\otimes I_H))_{k=1}^{\infty}
\ee
and can view it as an operator from $\cF(E)\otimes H$ to $H$. We get
$$\Phi_T(a)=b(T)(I_{\cF(E)}\otimes a)b(T)^* $$ for $a\in B(H)$.

Now compute
$$\Phi_T^2(a)=\Sigma_{k=1}^{\infty}T^{(k)}(X_k\otimes \Sigma_{l=1}^{\infty}T^{(l)}(X_l \otimes a)T^{(l) *})T^{(k) *}=$$  $$\Sigma_{k,l\geq 1}T^{(k)}(I_{E^{\otimes k}}\otimes T^{(l)})(X_k\otimes X_l\otimes a)(I_{E^{\otimes k}}\otimes T^{(l)})^*T^{(k) *}=$$   $$\Sigma_{m=2}^{\infty}T^{(m)}(\Sigma_{k+l=m} X_k\otimes X_l \otimes a)T^{(m) *}.$$
Continuing this way, we get
$$\Phi_T^n(a)=\Sigma_{m=n}^{\infty} T^{(m)}(\Sigma_{k_1+\cdots+k_n=m} (X_{k_1}\otimes \cdots \otimes X_{k_n})\otimes a)T^{(n) *}.$$

Write $X(m,n)=\Sigma_{k_1+\cdots+k_n=m} (X_{k_1}\otimes \cdots \otimes X_{k_n})\in B(E^{\otimes m})$ and let

 $b(T)^{(n)}$ be the row
$$b(T)^{(n)}=(T^{(k)}(X(k,n)\otimes I_H))_{k=1}^{\infty} $$
which can be viewed as an operator in $B(\cF(E)\otimes H,H)$, to get
$$\Phi_T^n(I)=b(T)^{(n)}b(T)^{(n) *}.$$
Note that, if $\Phi_T(I)\leq I$, then, for every $n$, $||\Phi_T^n(I)||\leq 1$ and $||b(T)^{(n) *}||\leq 1$. Thus, we get the following.

\begin{lemma}\label{bn}
Suppose $\Phi_T(I)\leq I$. Then $T$ is pure (that is, $\Phi_T^n(I)\searrow 0$) if and only if $||b(T)^{(n) *}h|| \rightarrow 0$ for every $h$ in a subset of $H$ that spans a dense subspace of $H$.
\end{lemma}

\section{$\cF(\cR)$ (weighted Drury-Arveson space)}

For $k\geq 1$, we write $$\cR_k^0=R_k(\bC^d)^{\circledS k} \subseteq (\bC^d)^{\otimes k}$$ (where $\circledS$ is the symmetric tensor product) and $\cR_0^0=\bC$.
Also write $$\cF(\cR):=\oplus_{k=0}^{\infty} \cR_k^0 \subseteq \cF(\bC^d) .$$

\begin{proposition}\label{Kzw}
For $z,w\in D(X,\bC) $, write
$$K^R(z,w)=\Sigma_{k=0}^{\infty}z^{(k)}R_k^2w^{(k) *}.$$
Then $K^R$ is a well defined positive kernel on $D(X,\bC)\times D(X,\bC)$ and the corresponding RKHS $H_{K^R}$ is isomorphic to $\cF(\cR)$ via the map $u$ that sends the kernel function $k^R_w$ to $\Sigma_{k=0}^{\infty} R_k w^{(k) *}$.

\end{proposition}

\begin{proof}
To show that $K^R(z,w)$ is well defined (that is, the series is convergent) for $z,w \in D(X,\bC)$, we note first that, for $z\in D(X,\bC)$, $\Phi_z:=\Sigma_{l=1}^{\infty}z^{(l)}X_lz^{(l) *}< 1$ and, thus, $\Sigma_l \Phi_z^l$ converges.

Now note that, for $l,m$ we have $z^{(l+m)}=z^{(l)}(I_{(\bC^d)^{(l)}}\otimes z^{(m)})$ and, thus, $z^{(l+m)}(X_l\otimes X_m)z^{(l+m) *}=z^{(l)}(I_{(\bC^d)^{(l)}}\otimes z^{(m)})(X_l\otimes X_m)(I_{(\bC^d)^{(l)}}\otimes z^{(m) *})z^{(l) *}=z^{(l)}(X_l\otimes z^{(m)}X_mz^{(m) *})z^{(l) *}$. It follows that
$$\Phi_z^2=\Sigma_{l,m}^{\infty} z^{(l+m)}(X_l\otimes X_m)z^{(l+m) *}=\Sigma_{k=1}^{\infty} z^{(k)}(\Sigma_{l+m=k} X_l\otimes X_m)z^{(k) *} .$$
Similarly, for $k\geq 1$,
$$\Phi_z^l=\Sigma_{k=1}^{\infty} z^{(k)}(\Sigma_{i_1+\cdots +i_l=k}X_{i_1}\otimes \cdots \otimes X_{i_l})z^{(k) *} $$ and, thus,
$$\Sigma_{k=0}^{\infty}z^{(k)}R_k^2z^{(k) *}=\Sigma_l \Phi_z^l $$ converges. Since this holds also for $w$, it follows that $K^R(z,w)$ is well defined. Positivity is clear and it is left to identify $H_{K^R}$ with $\cF(\cR)$.

For this, note that $H_{K^R}$ is spanned by the functions
$$k_w(z)=\Sigma_{k=0}^{\infty} z^{(k)}(R_k^2 w^{(k) *}) , \;\; z\in D(X,\bC) $$ where $w$ runs over $D(X,\bC)$. Now consider the linear map $u$ defined by $$u(k_w)=\Sigma_{k=0}^{\infty} R_k w^{(k) *} \in \cF(\cR).$$

We now need to show that $u$ preserves inner products. It will then follow that $u$ is well defined and can be extended to an isometry from $H_{R^K}$ to $\cF(\cR)$.
Since $H_{K^R}$ is generated by the functions $\{k_w:\; w\in D(X,\bC)\}$, it suffices to show that $u$ preserves inner products of these functions. So we compute, for $z,w\in D(X,\bC)$,
$$\langle k_z,k_w\rangle=K^R(z,w)=\Sigma_{m=0}^{\infty} z^{(m)}R_m^2w^{(m) *}=\Sigma_m \langle R_m z^{(m) *}1,R_mw^{(m) *}1\rangle_{(\bC^d)^{\otimes m}}$$  $$=\Sigma_m\langle R_m(z^*)^{\otimes m},R_m(w^*)^{\otimes m}\rangle_{(\bC^d)^{\otimes m}}.$$

Thus, $u$ is a well defined isometry into $\cF(\cR)$. To show that this map is surjective, note first that, for $\lambda \in \mathbb{T}$,

\be\label{lambda1}
 u(k_{\overline{\lambda} w})=\Sigma_{k=0}^{\infty} R_k \lambda^k w^{(k) *} \in \cF(\cR).
\ee

 We now define, for every $\lambda\in \mathbb{T}$, $W_{\lambda}\in B(\cF(\cR))=B(\Sigma_k\oplus \cR_k^0)$ by $W_{\lambda}=\Sigma_k \lambda^k Q_k$ (where $Q_k$ is the projection onto $\cR_k^0$). Then $\{W_{\lambda}\}$ is a strongly continuous group of unitaries and (\ref{lambda1}) shows that the range of $u$ is invariant under each $W_{\lambda}$.
 Since, for every $k\geq 0$, $Q_k=\int W_{\lambda}\lambda^{-k}d\lambda$ (in the strong operator topology), we see that the range of $u$ (which is a closed subspace) is invariant under $Q_k$. Since $u(k_w)=\Sigma_{k=0}^{\infty} R_k w^{(k) *} $, we find that, for every $k\geq 0$ and every $w\in D(X,\bC)$, $R_k w^{(k) *}$ lies in the range of $u$. Thus, for every $k\geq 0$, $R_k(\bC^d)^{\circledS k}$ is contained in the range of $u$. Since the range of $u$ is a closed subspace, $u$ is surjective.

\end{proof}


%



\begin{lemma}\label{cRk}
Given $T\in \overline{D}_c(X,H)$ and, for every $k\geq 1$, write $\cR_k:=\cR_k^0\otimes H$. Then
$$\mathcal{R}_k=(span\{(R_k \otimes A) T^{(k) *}h :\;A\in B(H),\; h\in H, T\in D_c(X,H) \})^{-}.$$

\end{lemma}
\begin{proof}
To prove that $(span\{(R_k \otimes A) T^{(k) *}h :\;A\in B(H),\; h\in H, T\in D_c(X,H) \})^{-}\subseteq \cR_k =R_k(\bC^d)^{\circledS k}\otimes H $ it suffices to show that, for every $A\in B(H)$, $h\in H$ and $T\in \overline{D}_{c}(X,H)$, $(I\otimes A)T^{(k) *}h \in (\bC^d)^{\circledS k}\otimes H $.

To prove it, we fix the following notation. $\{e_i\}_{i=1}^d$ is the standard orthonormal basis of $\bC^d$, for every word $\alpha\in \mathbb{F}_d^+$ of length $k$ we set $e_{\alpha}=e_{\alpha(1)}\otimes \cdots \otimes e_{\alpha(k)} \in (\bC^d)^{\otimes k}$ and, for a permutation $\pi \in S_k$ we write $\pi \alpha$ for the word obtained from $\alpha$ by applying $\pi$. Finally, we write $\tilde{e}_{\alpha}:=\Sigma_{\pi \in S_k} e_{\pi\alpha}\in (\bC^d)^{\circledS k}$.

Using the fact that $T$ is a commuting tuple, it is easy to check that $(I\otimes A)T^{(k) *}h$ is a linear combination of $\tilde{e}_{\alpha} \otimes AT_{\alpha}^*h$ and, thus, is contained in $ (\bC^d)^{\circledS k}\otimes H $.

For the converse inclusion take $A=I$ and $T=(\lambda_1I,\lambda_2I,\ldots,\lambda_dI)$ (for $\lambda_i\in \bC$ with $|\lambda_i|<1$). Then, for $h\in H$, $T^*h=\Sigma_{i=1}^d \overline{\lambda_i}e_i \otimes h$. Also,
$$(T^{(2)})^*h=(I_{\bC^d}\otimes T^*)(\Sigma_{i=1}^d \overline{\lambda_i}e_i \otimes h)=\Sigma_{i=1}^d \overline{\lambda_i}e_i \otimes T^*h=(\Sigma_{i=1}^d \overline{\lambda_i}e_i )\otimes (\Sigma_{i=1}^d \overline{\lambda_i}e_i )\otimes h.$$
Similarly,
$$(T^{(k)})^*h=(\Sigma_{i=1}^d \overline{\lambda_i}e_i)^{\otimes k} \otimes h.$$
Since $(\bC^d)^{\circledS k}$ is generated by vectors of the form $(\Sigma_{i=1}^d \overline{\lambda_i}e_i )^{\otimes k}$, we are done.

\end{proof}

We now set $Z_0=I$ and, for $k>0$,
\be\label{Zk}
Z_{k}=R_{k}^{-1}(I_{\bC^d}\otimes R_{k-1}).
\ee
 (See \cite[Equation (4.5)]{MSWS}).

On $\cF(\bC^d)$ we define, for every $1\leq i \leq d$, the operator $\tilde{W}_i$ by
\be\label{tildeW}
\tilde{W}_i \theta=Z_{k+1}(e_i \otimes \theta)
\ee
for $\theta \in (\bC^d)^{\otimes k}$. So that $\tilde{W}_i$ is a weighted shift mapping $(\bC^d)^{\otimes k}$ to $(\bC^d)^{\otimes (k+1)}$.

We also define the tuple $W=\{W_i\}$ on $\cF(\cR)$ by the compression of $\tilde{W}$ to $\cF(\cR)$, $$W_i=P_{\cF(\cR)}\tilde{W}_i|\cF(\cR).$$

\begin{lemma}\label{Wi}
The sequence $\{Z_k\}$ is bounded (so that $\tilde{W}_i$ is a bounded operator for each $i$) and
$$ \tilde{W}_i^* \cF(\cR) \subseteq \cF(\cR) .$$
Thus, $W_i^*=\tilde{W}_i^*|\cF(\cR)$
and, in fact,
\be\label{Wistar}
W_i^*R_kw^{(k) *}=\tilde{W}_i^*R_kw^{(k) *}=R_kL_i^*w^{(k) *}
\ee for $w\in D(X,\bC)$ where $L_i:(\bC^d)^{\otimes k} \rightarrow (\bC^d)^{\otimes (k+1)}$ is the operator $L_i\theta =e_i\otimes \theta$.

Also, for a word $\alpha$, $W_{\alpha}=P_{\cF(\cR)}\tilde{W}_{\alpha}|\cF(\cR)=P_{\cF(\cR)}\tilde{W}_{\alpha}$ and $W_{\alpha}^*=\tilde{W}_{\alpha}|\cF(\cR) $.
\end{lemma}
\begin{proof}
The boundedness of the sequence $\{Z_k\}$ follows from the proof of \cite[Theorem 4.5]{MSWS}. For the other statement, write $L_i:(\bC^d)^{\otimes k} \rightarrow (\bC^d)^{\otimes (k+1)}$ for the operator $L_i\theta =e_i\otimes \theta$ and compute (using Equation (\ref{Zk}) and the fact that, for every $k$, $L_i^*(\bC^d)^{\circledS (k+1)}\subseteq  (\bC^d)^{\circledS (k)}$)
$$\tilde{W}_i^*\cR_{k+1}^0=\tilde{W}_i^*R_{k+1}(\bC^d)^{\circledS (k+1)}=L_i^*(I_{\bC^d}\otimes R_k)R_{k+1}^{-1}R_{k+1}(\bC^d)^{\circledS (k+1)}=$$   $$L_i^*(I_{\bC^d}\otimes R_k)(\bC^d)^{\circledS (k+1)}=
 R_k L_i^*(\bC^d)^{\circledS (k+1)}\subseteq R_k (\bC^d)^{\circledS (k)}=\cR_k^0.
$$

\end{proof}

It follows from the boundedness of $\{Z_k\}$, Equation (\ref{Zk}) and the fact that $R_k$ is self adjoint, that
\be\label{Rk2}
I_{\bC^d}\otimes R_{k-1}^2=(I_{\bC^d}\otimes R_{k-1})(I_{\bC^d}\otimes R_{k-1})^*=R_kZ_kZ_k^*R_k\leq CR_k^2
\ee
where $C=\sup_k{||Z_k||^2}$.

\begin{lemma}\label{WinDcp}
$W=\{W_i\}\in \overline{D}_{cp}(X,\cF(\cR)) $.
\end{lemma}
\begin{proof}
It is shown in \cite[(5.2)]{MSWS} that $\Sigma_{k=1}^{\infty}\tilde{W}^{(k)}(X_k\otimes I_{\cF(\bC^d)})\tilde{W}^{(k) *} =I_{\cF(\bC^d)} -P_0$ where $P_0$ is the projection onto $\bC\subseteq \cF(\bC^d)$. Since $P_{\cF(\cR)}\tilde{W}^{(k)}=W^{(k)}$, we have $\Sigma_{k=1}^{\infty}W^{(k)}(X_k\otimes I_{\cF(\cR)})W^{(k) *} =I_{\cF(\cR)} -P_0 \leq I$. Thus $W\in \overline{D}(X,\cF(\cR))$.

To show that $W$ is a commutative tuple recall first that, for $\xi\in (\bC^d)^{\circledS (k+1)}$, $\tilde{W}_i^* R_{k+1}\xi=R_kL_i^*\xi$ (see the computation in the proof of Lemma~\ref{Wi}). Thus, for such $\xi$ and $i,j$,
$$W_i^*W_j^*R_{k+1}\xi=\tilde{W}_i^*\tilde{W}_j^*R_{k+1}\xi=\tilde{W}_i^*R_kL_i^*\xi=R_{k-1}L_j^*L_i^*\xi.$$ Since $L_i^*L_j^*=L_j^*L_i^*$ on $(\bC^d)^{\circledS (k+1)}$, $W_i^*W_j^*=W_j^*W_i^*$ for all $i,j$ implying that $W$ is a commutative tuple.

It is left to prove that it is pure. For this, we use Lemma~\ref{bn}. Since $span\{ R_l \xi^{\otimes l} :\; l\geq 0 ,\; \xi\in \mathbb{C}^d \}$ is dense in $\mathcal{F}(\cR)$, we need to show that
$$b(W)^{(n) *}R_l\xi^{\otimes l} \rightarrow_{n\rightarrow \infty} 0.$$
Since, for $m>l$, $W^{(m) *}R_l\xi^{\otimes l}=0$, we have
$$b(W)^{(n) *}R_l\xi^{\otimes l}=\Sigma_{m=n}^{\infty} \oplus (X(m,n)\otimes I_{\cF(\cR)})W^{(m) *}R_l\xi^{\otimes l}$$     $$=\Sigma_{m=n}^{l} \oplus (X(m,n)\otimes I_{\cF(\cR)})W^{(m) *}R_l\xi^{\otimes l} \rightarrow 0$$ as $n\rightarrow \infty$.

\end{proof}

Consider the reproducing kernel Hilbert space $H_{K^R}$ as in Proposition~\ref{Kzw}. Recall that its elements are scalar-valued functions defined on $D(X,\bC)$ and a function $f:D(X,\bC)\rightarrow \bC$ is called a \emph{multiplier} if for every $g\in H_{K^R}$, $fg$ also lies in $H_{K^R}$. In this case, we write $M_{f}$ for the operator that sends $g\in H_{K^R}$ to $fg$. By the closed graph theorem, $M_f$ is a bounded operator.

\begin{lemma}\label{Mzi}
	We keep the notation of Proposition~\ref{Kzw}.
	\begin{enumerate}
		\item[(1)] For every $1\leq i \leq d$, write $z_i$ for the function on $D(X,\bC)$ defined by $z_i(w)=w_i$. Then $z_i$ is a multiplier of $H_{K^R}$ and we have
\be\label{MW}
W_i u=u M_{z_i}.
\ee
		\item[(2)]  $\cap Ker(M_{z_i}^*-\overline{w}_i)=\bC k^R_w$ for $w\in D(X,\bC)$.
	\end{enumerate}
\end{lemma}
\begin{proof} To prove part (1) we first show that each $z_n$ is a multiplier. For this, we need to show that there is a positive constant $C$ such that the kernel $K'(z,w):=CK^R(z,w)-z_nK^R(z,w)w_n^*$ is positive (\cite[Theorem 6.28]{PR}). Note that, since $K^R(z,w)$ is a positive kernel, so is the kernel $K_n(z,w):=z_nK(z,w)w_n^*$ for every $n$. Thus, it will suffice to show that the kernel $K''(z,w):=CK^R(z,w)-\Sigma_nz_nK^R(z,w)w_n^*$ is positive for some $C$. So, we fix $w(1),\ldots,w(m)$ in $D(X,\bC)$ and consider the matrix $(K''(w(i),w(j))_{i,j}$. The $i,j$ entry is
$$K''(w(i),w(j))=\Sigma_{k=0}^{\infty}w(i)^{(k)}R_k^2w(j)^{(k) *}-\Sigma_n\Sigma_{k=0}^{\infty}w(i)_nw(i)^{(k)}R_k^2w(j)^{(k) *}\overline{w(j)_n}. $$
Write $A(i,j)=\Sigma_n\Sigma_{k=0}^{\infty}w(i)_nw(i)^{(k)}R_k^2w(j)^{(k) *}\overline{w(j)_n}$.
But then  $$A(i,j)=\Sigma_{k=0}^{\infty} w(i)^{(k+1)}(I\otimes R_{k}^2)w(j)^{(k+1) *} $$ and it follows from (\ref{Rk2}) that $A$ is smaller than $B$ where $$B(i,j)= C \Sigma_{k=0}^{\infty} w(i)^{(k+1)} R_{k+1}^2 w(j)^{(k+1) *}$$ for $C=\sup{||Z_k||^2}$. Thus, $K''$ is a positive kernel and each $z_i$ is a multiplier. It follows that $M_{z_i}^*k_w=\overline{w_i}k_w$.

To prove (\ref{MW}), we compute $$uM_{z_i}^*u^*(\Sigma_{k=0}^{\infty}R_kw^{(k)*})=uM_{z_i}^*k_w^R=\overline{w_i}u(k_w^R)=\overline{w_i}\Sigma_{k=0}^{\infty}R_kw^{(k)*}.$$
But, using (\ref{Wistar}) and the fact that $W_i^*$ vanishes on the $0$th term ($\bC$), we get $$uM_{z_i}^*u^*(\Sigma_{k=0}^{\infty}R_kw^{(k)*})=W_i^*(\Sigma_{k=0}^{\infty}R_kw^{(k)*}).$$

	For part (2), note that
$\cap Ker(M_{z_i}^*-\overline{w}_i)\supseteq\bC k^R_w$ is true in general. For the other inclusion, assume $\Sigma R_k\xi_k\in \cap Ker(W_i^*-\overline{w}_i)$ for $\xi_k\in (\bC^d)^{\circledS k}\subseteq (\bC^d)^{\otimes k}$.
We compute
$$\overline{w}_i(\Sigma R_k\xi_k)=W_i^*(\Sigma R_k\xi_k)=\tilde{W}_i^*(\Sigma R_k\xi_k)=\Sigma L_i^*Z_k^*R_k\xi_k=\Sigma L_i^*(I_{\bC^d}\otimes R_{k-1})\xi_k=$$  $$\Sigma R_{k-1}L_i^*\xi_k.$$
Thus, for every $k$ and every $i$, $\overline{w}_iR_k\xi_k=R_kL_i^*\xi_{k+1}$ and, since $R_k$ is invertible,
$$\overline{w}_i\xi_k=L_i^*\xi_{k+1}.$$
It is now easy to check that, setting $\xi_0=1$, there is only a unique solution (or use the fact that it is known for the Drury-Arveson space and what we did above is reducing it to this case) and this proves (2).
\end{proof}

\section{Dilations and invariant subspaces}

For $T\in \overline{D}_{cp}(X,H)$, we write $\Delta_*(T)=(I_H-\Phi_T(I))^{\frac{1}{2}}$ and $\cD_*:=\overline{\Delta_*(T)(H)}$. We also define the operator
$$\Pi(T): H\rightarrow \mathcal{F}(\mathcal{R})\otimes \cD_* $$ by
$$\Pi(T)=(I_{\cF(\cR)}\otimes \Delta_*(T))(I_H \;\; (R_1 \otimes I_H)T^* \;\; (R_2\otimes I_H)(T^{(2)})^* \; \; \ldots )^T .$$
That is,
$$\Pi(T)h=(I_{\cF(\cR)}\otimes \Delta_*(T))(h \oplus  (R_1 \otimes I_H)T^*h \oplus (R_2\otimes I_H)(T^{(2)})^*h \oplus \ldots )$$
for $h\in H$. (See \cite[Definition 5.3]{MSWS}, \cite{MS2009} and \cite{Po1999} where this map is referred to as the Poisson kernel). It generalizes the operator $\Pi_c$ defined in \cite{BEKS} which is shown there to be a dilation of a commuting tuple to a Hilbert space valued Drury-Arveson space. Indeed, if $R_k=I$ for all $k$, $\cF(\cR)$ is the Drury-Arveson space as was shown in Lemma~\ref{Rk}.

\begin{lemma}\label{propPi}
For every $T\in \overline{D}_{cp}(X,H)$,  the map $\Pi(T): H\rightarrow \mathcal{F}(\mathcal{R})\otimes \cD_* $ is an isometry and satisfies,
$$\Pi(T)T_i^* = (W_i^*\otimes I_{\cD_*})\Pi(T) $$
for every $1\leq i \leq d$.

\end{lemma}
\begin{proof}
The proof that $\Pi(T)$ is an isometry is the same as in \cite[Lemma 5.4]{MSWS}(using the assumption that $T$ is pure) and the rest follows from the proof of \cite[Lemma 5.5  1.]{MSWS}. To see this, replace $\mathfrak{z}$ there by $T$, $K(\mathfrak{z})$ there by $\Pi(T)$ here and $\xi $ there by $e_i$ here.

\end{proof}


The following result generalizes \cite[Theorem 4.1]{BEKS}.

\begin{theorem}\label{Tinvariant}
Let $T\in \overline{D}_{cp}(X,H)$ and $\cS \subseteq H$ be a subspace of $H$. Then $\cS$ is a joint $T$-invariant subspace if and only if there exists a Hilbert space $\cD$ and a partial isometry $\Pi\in B(\cF(\cR)\otimes \cD, H)$ with $T_i\Pi  = \Pi(W_i\otimes I_{\cD})$ for all $1\leq i \leq d$ and $$\cS=\Pi(\cF(\cR)\otimes \cD).$$

\end{theorem}
\begin{proof}
It is clear that if $\cS=\Pi(\cF(\cR)\otimes \cD)$ and $T_i\Pi  = \Pi(W_i\otimes I_{\cD})$ , then $\cS$ is a joint $T$-invariant subsplace of $H$.

For the other direction we consider the restriction $T|\cS\in B(\cS)$. Write $\iota_{\cS}:\cS\rightarrow H$ for the inclusion map and $P_{\cS}$ for its adjoint, the (orthogonal) projection onto $\cS$. It is easy to check that, for every $k\geq 1$, $(T|\cS)^{(k)}=T^{(k)}|E^{\otimes k}\otimes \cS$ and, thus,
$$\Phi_{T|\cS}(a)=\Sigma_{k=1}^{\infty} T^{(k)}(X_k\otimes P_{\cS}aP_{\cS})T^{(k) *}.$$
It follows that $T\in \overline{D}_{c}(X,H)$ and, to apply Lemma~\ref{propPi} to $T|\cS$, we need to check that it is also pure. Now
$$\Phi_{T|\cS}^2(a)=\Sigma_{k=1}^{\infty} T^{(k)}(X_k\otimes P_{\cS}\Sigma_{m=1}^{\infty} T^{(m)}(X_m\otimes P_{\cS}aP_{\cS})T^{(m) *}P_{\cS})T^{(k) *}=$$   $$\Sigma_{k,m}T^{(k)}(X_k\otimes P_{\cS} T^{(m)}(X_m\otimes P_{\cS}aP_{\cS})T^{(m) *}P_{\cS})T^{(k) *} $$ Since $\cS$ is $T$-invariant, we have $P_{\cS}T^{(m)}(I_{E^{\otimes m}}\otimes P_{\cS})=T^{(m)}(I_{E^{\otimes m}}\otimes P_{\cS})$ and, therefore,
$$\Phi_{T|\cS}^2(a)=\Sigma_{k,m}T^{(k)}(X_k\otimes  T^{(m)}(X_m\otimes P_{\cS}aP_{\cS})T^{(m) *})T^{(k) *} =$$   $$\Sigma_{k,m} T^{(k+m)}(X_k\otimes X_m\otimes P_{\cS}aP_{\cS})T^{(k+m) *}.$$
The same computation for $T$ (instead of $T|\cS$) yields
$$\Phi_{T}^2(a)=\Sigma_{k,m} T^{(k+m)}(X_k\otimes X_m\otimes a)T^{(k+m) *}$$ and, thus $$\Phi_{T|\cS}^2(a)=\Phi_{T}^2(P_{\cS}aP_{\cS}).$$
Continuing in this way (see also the computation in \cite[Proof of Theorem 4.5]{MSWS}), we get, fot every $n\geq 1$,
$$\Phi_{T|\cS}^n(a)=\Phi_{T}^n(P_{\cS}aP_{\cS})$$ and, since $T$ is pure, so is $T|\cS$.

The rest of the argument proceeds as in \cite[Proof of Theorem 4.1]{BEKS}.
We use Lemma~\ref{propPi} to get an isometry $\Pi_{\cS}:\cS \rightarrow \cF(\cR)\otimes \cD$ that satisfies $\Pi_{\cS}P_{\cS}T_i^* = (W_i^*\otimes I_{\cD})\Pi_{\cS} $. Finally, let
$$\Pi:=\iota_{\cS} \circ \Pi_{\cS}^*: \cF(\cR)\otimes \cD \rightarrow H $$ be the required map.

\end{proof}

\section{Invariant subspaces in reproducing kernel correspondences}

With the notation set up above, we now consider the kernel
\be
K^R(V,W)(a)=\Sigma_{k=0}^{\infty}V^{(k)}(R_k^2\otimes a)W^{(k) *}=\Sigma_{|\alpha|=|\beta|} V_{\alpha}r^2_{\alpha,\beta}a W_{\beta}^*
\ee where $V,W\in D(X,H)$ and $a\in B(H)$. This defines a completely positive maps-valued kernel on $D(X,H)\times D(X,H)$ with values in the bounded maps on $B(H)$.
In \cite[Theorem 4.5]{MSWS} it was shown that it is well defined (where the sum converges in the norm topology) and in \cite{GThesis} J. Good studied this kernel (in the more general context of $W^*$-correspondence $E$ over a von Neumann algebra $M$) and the reproducing kernel $W^*$-correspondence associated to this kernel. She denoted it $\cH^2(X,\sigma)$ (where $\sigma$ is the representation of $M$ on $H$).

Since, here, we are interested in the commuting tuples we write $K^R_c(V,W)$ for the restriction of $K^R$ to $D_c(X,H)\times D_c(X,H)$. 


Now we consider a kernel
$$ K: D_c(X,H)\times D_c(X,H) \rightarrow B_*(B(H),B(H))$$
that is completely positive. That means that, given $\{V_i\}_{i=1}^n$ with $V_i\in D_c(X,H)$, the map
$$(a_{i,j}) \mapsto (K(V_i,V_j)(a_{i,j})) $$
on $M_n(B(H))$ is completely positive.

Associated with such a kernel one defines a $W^*$-correspondence $E_K$ over $B(H)$.
The details of the following statements can be found in \cite[Chapter 3]{MThesis} or in \cite{GThesis}.

 In general, map-valued cp kernels  are functions $K:\Sigma \times \Sigma \rightarrow B_*(N,L)$ where $\Sigma$ is a set and $N$, $L$ are $W^*$-algebras and it is completely positive in the sense that, given $n$ and points $z_1,\ldots,z_n$ in $\Sigma$, the matrix (of maps) $(K(z_i,z_j))_{i,j=1}^{n}$ represents a (normal) completely positive map from $M_n(N)$ to $M_n(L)$.
 Associated with such a kernel one gets an $N-L$ $W^*$-correspondence $E_K$. The elements of $E_K$ are functions $f:\Sigma\rightarrow B(N,L)$.

 Conversely, given such a correspondence, with reproducing property as below, one gets a completely positive maps-valued kernel.

Although the details can be found in \cite{MThesis} and in \cite{GThesis}, we sketch those details about the construction of $E_K$ that we shall need. Assume a map-valued cp kernel $K:\Sigma \times \Sigma \rightarrow B_*(N,L)$ is given. We define functions $k_{(a,w)}:\Sigma \rightarrow B_*(N,L)$, for $a\in N$ and $w\in \Sigma$ by $$
k_{(a,w)} (z)(b)=K(z,w)(ba^*)$$
for $z \in \Sigma$ and $b \in N$. These functions generate $E_K$ as an $N-L$ correspondence where the left action of $N$ on these functions is defined by
\be\label{leftmult}
d\cdot k_{(a,w)} =k_{(ad^*,w)}
\ee and the $L$-valued inner product of these functions is given by $$ \langle k_{(b,z)},k_{(a,w)}\rangle = k_{(a,w)}(z)(b)= K(z,w)(ba^*) .$$

Thus, the elements of $E_K$ are functions $f:\Sigma\rightarrow B_*(N,L)$ and $E_K$ is generated, as a $W^*$-module by $\{k_{(a,w)}\cdot c: a\in N, w\in \Sigma, c\in L\}$.
The left action of $N$ on $f\in E_K$ is defined by $(a\cdot f)(z)(b)=f(z)(ba)$.

For $f\in E_K$, we have $$f(z)(a)=\langle k_{(a,z)},f\rangle $$ so these are the kernel functions that induce point evaluations. We also have
$$\langle k_{(b,z)}\cdot c,k_{(a,w)}\cdot d\rangle = c^*K(z,w)(ba^*)d \in L .$$

The following definition and theorem can be found in \cite[Definition 46, Lemma 47, Theorem 48]{MThesis}.

\begin{definition}\label{multiplier}
Let $N$ and $L$ be $W^*$-algebras and let $E_K$ be a reproducing kernel $W^*$-correspondence from $N$ to $L$ with associated cp kernel $K:\Sigma \times \Sigma \rightarrow B_*(N,L)$. Then a function $\phi:\Sigma \rightarrow L$ is called a multiplier of $E_K$ if for each $f\in E_K$, the function $\phi f$, defined by $(\phi f)(z)(b)=\phi(z)f(z)(b)$ is in $E_K$.
\end{definition}

\begin{theorem}\label{charmultiplier}
Let $N$ and $L$ be $W^*$-algebras and let $E_K$ be a reproducing kernel $W^*$-correspondence from $N$ to $L$ with associated cp kernel $K:\Sigma \times \Sigma \rightarrow B_*(N,L)$. Then
\begin{enumerate}
\item[(1)] If $\phi$ is a multiplier of $E_K$ then the map $M_{\phi}:E_K\rightarrow E_K$, defined by $M_{\phi}f=\phi f$, is in $\cL(E_K)\cap \varphi_{E_K}(N)'$.
\item[(2)] If $\phi$ is a multiplier of $E_K$ then, for all $(a,w)\in N\times \Sigma$, $M^*_{\phi}k_{a,w}=k_{a,w}\phi(w)^*$.
\item[(3)] A function $\phi:\Sigma \rightarrow L$ is a multiplier of $E_K$ with $||M_{\phi}||\leq 1$ if and only if the map $K_{\phi}:\Sigma \times \Sigma \rightarrow B_*(N,L)$ defined by $K_{\phi}(w,z)=(id - Ad(\phi(w),\phi(z)))\circ K(w,z)$ is a cp kernel.

\end{enumerate}

\end{theorem}

Assume from now on that $N=L=B(H)$, $\Sigma=D_c(X,H)$ and $K$ is such a map-valued kernel (of maps on $B(H)$).

\begin{definition}\label{XH}
Suppose $K$ is as above and, for every $i$, the map $S_i:D_c(X,H)\rightarrow B(H)$, defined by $S_i(T)=T_i$, is a bounded multiplier of $E_K$. If $S:=(S_i)$ satisfies
$$\Sigma_{k=1}^{\infty} M_S^{(k)} (X_k\otimes I_{E_K})M_S^{(k) *} \leq I$$
We shall say that $E_K$ is an \emph{(X,H)-contractive reproducing kernel correspondence}.
\end{definition}

\begin{proposition}
Suppose $\Phi_{M_S\otimes_{B(H)}H}(I)\leq I$ then $\Phi_{M_S\otimes_{B(H)}H}^n(I)\searrow 0$.
Thus, if $E_K$ is an (X,H)-contractive reproducing kernel correspondence then
\be\label{Xcontractivecondition}
M_S \otimes_{B(H)}I_H \in \overline{D}_{cp}(X,E_K\otimes_{B(H)}H).
\ee
\end{proposition}
\begin{proof}
For simplicity, we will write here $G$ for $E_K\otimes_{B(H)}H$, $M_i$ for $M_{S_i}\otimes H\in B(G)$ and $M=(M_1,\ldots,M_d)$ for $M_S\otimes I_H$.
By Lemma~\ref{bn}, we need to show that $||b(M)^{(n) *}(k_{a,T}\cdot c \otimes_{B(H)} h)||\rightarrow_n 0$ for $a,c\in B(H)$, $T\in D_c(X,H)$ and $h\in H$ (where $k_{a,T}$ is the kernel function in $E_K$ and $b(M)$ is the row defined in (\ref{bT}) with $M$ in place of $T$).

Recall that $S_i$ is the $i$th coordinate function on $D_c(X,H)$ and $M_{S_i}$ is the corresponding multiplier. Thus $M_{S_i}^*k_{a,T}\otimes_{B(H)}f=k_{a,T}\cdot T_i^*\otimes_{B(H)}f=k_{a,T}\otimes_{B(H)}T_i^*f$.
It follows that $$M_S^*k_{a,T} \otimes_{B(H)}cf=M_S^*k_{a,T}\cdot c \otimes_{B(H)}f=k_{a,T}\otimes_{B(H)}T^*cf.$$
Continuing this way, we can write, for every $m$,
$$M^{(m)*}(k_{a,T}\cdot c \otimes_{B(H)} h)=M_S^{(m)*}k_{a,T} \otimes_{B(H)}cf=k_{a,T}\otimes_{B(H)}T^{(m)*}cf$$
and
$$b(M)^{(n)*}(k_{a,T}\cdot c \otimes_{B(H)} h)=\Sigma_{m=n}^{\infty}\oplus (X(m,n)\otimes I_G)M^{(m)*}(k_{a,T}\cdot c \otimes_{B(H)} h)=$$ $$k_{a,T}\otimes_{B(H)}(\Sigma_{m=n}^{\infty}\oplus (X(m,n)\otimes I_G)T^{(m)*} ch)=k_{a,T}\otimes_{B(H)}b(T)^{(n)*}ch .$$
But $||b(T)^{(n)*}ch||^2=\langle \Phi_T^n(I)ch,ch\rangle$ and, since $T\in D(X,H)$, $\Phi_T^n(I)\searrow 0$ by \cite[Lemma 5.4 (4)]{MSWS} and this completes the proof.

\end{proof}



\begin{proposition}\label{Lambda}
\begin{enumerate}
\item[(1)] For every $1\leq i \leq d$, $S_i$ is a multiplier of $E_{K^R_c}$.
\item[(2)] There is an isomorphism
$$\Lambda: E_{K^R_c}\otimes_{B(H)}H \rightarrow \cF(\cR)\otimes H$$
 that satisfies
 \begin{enumerate}
\item[a.]\be\label{intertwineWi}
\Lambda^*(W_i\otimes I_H)\Lambda =M_{S_i}\otimes I_H
\ee for every $1\leq i \leq d$.
\item[b.] For $b\in B(H)$, $$ \Lambda (\varphi_{E_{K^R_c}}(b)\otimes I_H)=(I_{\cF(\cR)}\otimes b)\Lambda $$ where $\varphi_{E_{K^R_c}}$ is the left action of $B(H)$ on $E_{K^R_c}$.
    \end{enumerate}
\item[(3)] $E_{K^R_c}$ is an $(X,H)$-contractive reproducing kernel correspondence.

\end{enumerate}
\end{proposition}
\begin{proof}
Given $T\in D_c(X,H)$ we have, for $k\geq 1$ and $1\leq i \leq d$,
\be\label{E1}
T^{(k+1)}L_i=T_iT^{(k)}
 \ee
 where $L_i (f_1\otimes \cdots \otimes f_k \otimes h)=e_i\otimes f_1\otimes \cdots \otimes f_k \otimes h$ ($f_j\in E=\bC^d$, $h\in H$). To check this, just compute $T^{(k+1)}L_i (f_1\otimes \cdots \otimes f_k \otimes h)=T^{(k+1)}(e_i\otimes f_1\otimes \cdots \otimes f_k \otimes h)=T(I_E\otimes T^{(k)})(e_i\otimes f_1\otimes \cdots \otimes f_k \otimes h)=T(e_i\otimes T^{(k)}( f_1\otimes \cdots \otimes f_k \otimes h))=T_iT^{(k)}( f_1\otimes \cdots \otimes f_k \otimes h)$.

 To prove that $S_i$ is a multiplier of $E_{K_c^R}$, consider, for $T,V\in D_c(X,H)$ and $a,b\in B(H)$,
 $$S_i(T)K^R_c(T,V)(a^2)S_i(V)^*=\Sigma_{k=0}^{\infty} T_iT^{(k)}(R_k^2\otimes ba^*)V^{(k) *}V_i^*=$$ $$\Sigma_{k=0}^{\infty} T^{(k+1)}L_i(R_k^2\otimes ba^*)L_i^*V^{(k+1) *}.$$
 Using (\ref{Zk}) we have, for $f_i\in E=\bC^d$ and $h\in H$,
 $L_i(R_k \otimes a)(f_1\otimes \cdots \otimes f_k\otimes h)=e_i\otimes R_k(f_1\otimes \cdots \otimes f_k)\otimes ah=(I_E\otimes R_k)(e_i\otimes f_1\otimes \cdots \otimes f_k)\otimes ah=(R_{k+1}Z_{k+1} \otimes a)L_i (f_1\otimes \cdots \otimes f_k \otimes h)$. Thus
 \be\label{E2}
 L_i(R_k \otimes a)=(R_{k+1}Z_{k+1} \otimes a)L_i
 \ee and therefore
 \be\label{E3}
 S_i(T)K^R_c(T,V)(ba^*)S_i(V)^* \ee  $$=\Sigma_{k=0}^{\infty} T^{(k+1)}(R_{k+1}Z_{k+1} \otimes b)L_i L_i^*(Z_{k+1}^*R_{k+1}\otimes a^*)V^{(k+1) *}
    $$
 Now fix $V_1,\ldots,V_n$ in $D_c(X,H)$ and $a_1,\ldots,a_n$ in $B(H)$ and consider the $n\times n$ matrix
 $$(S_i(V_l)K^R_c(V_l,V_m)(a_la_m^*)S_i(V_m)^*)_{l,m}$$ which, by (\ref{E3}), is equal to the matrix
 $$(\Sigma_{k=0}^{\infty} V_l^{(k+1)}(R_{k+1}Z_{k+1} \otimes a_l)L_i L_i^*(Z_{k+1}^*R_{k+1}\otimes a_m^*)V_m^{(k+1) *})_{l,m}$$ $$\leq (\Sigma_{k=0}^{\infty} V_l^{(k+1)}(R_{k+1}Z_{k+1}Z_{k+1}^*R_{k+1}\otimes a_l a_m^*)V_m^{(k+1) *})_{l,m}$$
 $$\leq \sup_k||Z_k||^2 (\Sigma_{k=0}^{\infty} V_l^{(k+1)}(R_{k+1}^2\otimes a_l a_m^*)V_m^{(k+1) *})_{l,m}$$ $$\leq \sup_k||Z_k||^2(K_c^R(V_l,V_m)(a_la_m^*))_{l,m}$$
      proving that $S_i$ is a multiplier. In fact, we see that $||M_{S_i}||\leq \sup||Z_k||$ ($<\infty$). This proves (1).

 For (2), we define $\Lambda$ on generators by
 $$\Lambda(k^R_{a,T}\cdot c \otimes h)=\Sigma_{k=0}^{\infty}\oplus (R_k\otimes a^*)T^{(k)*}ch $$ for $a,c\in B(H)$ and $T\in D_c(X,H)$. By Lemma~\ref{cRk}, $(R_k\otimes a^*)T^{(k)*}ch\in \cR_k\subseteq \cF(\cR)\otimes H$. We now compute, for $a,b,c,d\in B(H)$, $T,L\in D_c(X,H)$ and $h,g \in H$,
 $$\Sigma_{k=0}^{\infty}\langle (R_k\otimes a^*)T^{(k)*}ch ,(R_k\otimes b^*)L^{(k)*}dg \rangle=\Sigma_k \langle h,c^* T^{(k)} (R_k^2\otimes ab^*)L^{(k)*}dg\rangle$$  $$=\langle h,c^* K^R(T,L)(ab^*)dg \rangle.$$ This shows that the series defining $\Lambda(k_{a,T}\cdot c \otimes h)$ converges and, thus, belongs to $\Sigma_k\oplus \cR_k=\cF(\cR)\otimes H$. We also get $$\langle \Lambda(k^R_{a,T}\cdot c \otimes h),\Lambda(k^R_{b,L}\cdot d \otimes g)\rangle=\langle h,c^* K^R(T,L)(ab^*)dg \rangle$$   $$=\langle h,\langle k^R_{a,T}\cdot c,k^R_{b,L}\cdot d\rangle g\rangle=\langle  k^R_{a,T}\cdot c \otimes h,k^R_{b,L}\cdot d \otimes g\rangle.$$ Thus, $\Lambda$ is a well defined isometry into $\cF(\cR)\otimes H$. To show that this map is surjective, note first that, for $\lambda \in \mathbb{T}$,
 \be\label{lambda}
 \Lambda(k^R_{a,\overline{\lambda}T}\cdot c \otimes h)=\Sigma_{k=0}^{\infty}\oplus \lambda^k(R_k\otimes a^*)T^{(k)*}ch.
\ee
 We now define, for every $\lambda\in \mathbb{T}$, $W_{\lambda}\in B(\cF(\cR)\otimes H)=B(\Sigma_k\oplus \cR_k)$ by $W_{\lambda}=\Sigma_k \lambda^k Q_k$ (where $Q_k$ is the projection onto $\cR_k$). Then $\{W_{\lambda}\}$ is a strongly continuous group of unitaries and (\ref{lambda}) shows that the range of $\Lambda$ is invariant under each $W_{\lambda}$.
 Since, for every $k\geq 0$, $Q_k=\int W_{\lambda}\lambda^{-k}d\lambda$ (in the strong operator topology), we see that the range of $\Lambda$ (which is a closed subspace) is invariant under $Q_k$. Since $\Lambda(k^R_{a,T}\cdot c \otimes h)=\Sigma_{k=0}^{\infty}\oplus (R_k\otimes a^*)T^{(k)*}ch $ for $a,c\in B(H)$ and $T\in D_c(X,H)$, we find that, for every $k\geq 0$ and every $a,c,T$ as above, $(R_k\otimes a^*)T^{(k)*}ch$ is in the range of $\Lambda$. It follows from Lemma~\ref{cRk} that the range of $\Lambda$ contains $\cR_k$, for every $k$, and, thus, $\Lambda$ is surjective.



To prove (\ref{intertwineWi}) we compute.
$$\Lambda(M_{S_i}^*\otimes I_H)\Lambda^*(\Sigma_{k=0}^{\infty} (R_k\otimes a^*)T^{(k)*}ch)=$$
$$ \Lambda(M_{S_i}^*\otimes I_H)( k_{a,T}\cdot c \otimes h)=\Lambda( k_{a,T}\cdot T_i^*c \otimes h)$$ $$=\Sigma_{k=0}^{\infty}
(R_k\otimes a^*)T^{(k) *}T_i^*c h=\Sigma_{k=0}^{\infty}(R_k\otimes a^*)L_i^*T^{(k+1) *}c h$$  $$= \Sigma_{k=0}^{\infty} L_i^*(Z_{k+1}^*R_{k+1}\otimes a^*)T^{(k+1) *}c h=$$  $$(W_i^*\otimes I_H)\Sigma_{k=0}^{\infty} (R_{k+1}\otimes a^*)T^{(k+1) *}c h =(W_i^*\otimes I_H)\Sigma_{k=0}^{\infty} (R_{k}\otimes a^*)T^{(k) *}c h
$$ where we used (\ref{E1}), (\ref{E2}) and, in the last equality, the fact that $W_i^*\otimes I_H$ vanishes on $H$. This proves (\ref{intertwineWi}) and completes the proof of part (2a.).

To prove (2b.) we use the definition of $\Lambda$ and (\ref{leftmult}) and compute
$$\Lambda(\varphi_{E_{K^R_c}}(b)\otimes I_H)(k^R_{a,T}\cdot c \otimes h)=\Lambda(k^R_{ab^*,T}\cdot c \otimes h)=\Sigma_{k=0}^{\infty}\oplus (R_k\otimes ba^*)T^{(k)*}ch$$ $$=(I_{\cF(\cR)}\otimes b)\Sigma_{k=0}^{\infty}\oplus (R_k\otimes a^*)T^{(k)*}ch=(I_{\cF(\cR)}\otimes b)\Lambda(k^R_{a,T}\cdot c \otimes h).$$

To prove part (3) we have to show that
$$\Sigma_{k=1}^{\infty} M_S^{(k)} (X_k\otimes I_{E_{K^R_c}})M_S^{(k) *} \leq I.$$ But, using part (2), it suffices to prove that $$\Sigma_{k=1}^{\infty} W^{(k)} (X_k\otimes I_{\cF(\cR)})W^{(k) *} \leq I$$
and this follows from Lemma~\ref{WinDcp}.
\end{proof}

For the next result, recall the notion of $\sigma$-dual of a correspondence. Let $F$ be a $W^*$-correspondence over a $W^*$-algebra $M$ and let $\sigma$ be a normal representation on a Hilbert space $H$. The $\sigma$-dual, $F^{\sigma}$ is a $W^*$-correspondence over the $W^*$-algebra $\sigma(M)'$ defined by
$$F^{\sigma}=\{Y:H \rightarrow F\otimes_{\sigma}H \;:\; Y\sigma(b)=(\varphi_F(b)\otimes I_H)Y,\; b\in M\}.$$
(See \cite{MSHardy} for the details). The $\sigma(M)'$-valued inner product is defined by $\langle Y_1,Y_2 \rangle=Y_1^*Y_2$ and the bimodule structure is defined by
$$a\cdot Y \cdot b=(I_F \otimes a)\circ Y\circ b$$ for $a,b\in \sigma(M)'$.

\begin{proposition}\label{dual}
Let $E_{K_c^R}$ be the correspondence over $B(H)$ associated with the kernel $K_c^R$ and $\sigma$ be the natural representation of $B(H)$ on $H$. Then
$$E_{K_c^R}^{\sigma}\cong \cF(\cR)$$ where the isomorphism is an isomorphism of Hilbert spaces.	
	
	\end{proposition}
\begin{proof}
	We define the map $\Psi:\cF(\cR) \rightarrow E_{K_c^R}^{\sigma}$ by
	$$\Psi (\xi)=\Lambda^{-1}\circ L_{\xi}\;\;\;,\; \xi\in \cF(\cR)$$ where $L_{\xi}:H\rightarrow \cF(\cR)\otimes H$ is defined by $L_{\xi}h=\xi \otimes h\in \cF(\cR)\otimes H$ and $\Lambda$ is the map of Proposition~\ref{Lambda}.
	
	To prove that $\Psi(\xi)$ (for $\xi\in \cF(\cR)$) lies in $E_{K_c^R}^{\sigma}$, fix $h\in H$ and $b\in B(H)$ and compute
	$$\Psi(\xi)\sigma(b)h= \Lambda^{-1}\circ L_{\xi}bh=\Lambda^{-1}(\xi \otimes bh)=\Lambda^{-1}(I_{\cF(\cR)}\otimes b)(\xi\otimes h)$$ and, using Proposition~\ref{Lambda}(2b.), this is equal to $$( \varphi_{E_{K_c^R}}(b)\otimes I_H)\Lambda^{-1}(\xi \otimes h)=( \varphi_{E_{K_c^R}}(b)\otimes I_H)\Psi(\xi)h$$ proving that $\Psi(\xi)\in E_{K_c^R}^{\sigma}$.
	
	Linearity of $\Psi$ is obvious.
	
	To show that $\Psi$ is surjective, fix $Y\in E_{K_c^R}^{\sigma}$ and write $\eta=\Lambda \circ Y:H \rightarrow \cF(\cR)\otimes H$. For $h\in H$ and $b\in B(H)$ we get (using Proposition~\ref{Lambda} (2b.)) that $\eta bh=(\Lambda \circ Y)bh=\Lambda (\varphi_{E_{K^R_c}}(b)\otimes I_H)Yh=(I_{\cF(\cR)}\otimes b)\eta h$. Thus, for $b\in B(H)$,
	\be\label{eta}
	\eta bh=(I_{\cF(\cR)}\otimes b)\eta h.
	\ee
	Now fix a unit vector $h_0\in H$ and write $\eta h_0=\Sigma_i \xi_i\otimes h_i$ (for $\xi_i\in \cF(\cR)$ and $h_i\in H$). Write $p_0$ for the projection onto $\bC h_0$ and use (\ref{eta}) to get $\eta h_0=\eta p_0h_0=\Sigma_i \xi_i \otimes p_0h_i$. Thus, we can write $\eta h_0=\xi \otimes h_0$ for some $\xi \in \cF(\cR)$. Now apply (\ref{eta}) again, with arbitrary $h\in H$ and $v$ which is the rank one operator mapping $h_0$ to $h$, to get $$\eta h=\eta vh_0=(I_{\cF(\cR)}\otimes v)\eta h_0=\xi \otimes vh_0=\xi \otimes h$$ proving that $\eta=L_{\xi}$ and $\Psi(\xi)=Y$.
	
	To complete the proof we need to show that $\Psi$ is an isometry and this follows from the fact that $\Lambda^{-1}$ is an isometry and from $\langle L_{\xi_1},L_{\xi_2}\rangle = L_{\xi_1}^*L_{\xi_2}=\langle \xi_1, \xi_2 \rangle$.
\end{proof}


The following two theorems present our version of the Beurling-Lax-Halmos Theorem for reproducing kernel correspondences.

\begin{theorem}\label{BLH}
Suppose $E_K$ is an $(X,H)$-contractive reproducing kernel correspondence, $G$ is a Hilbert space and $\cS\subseteq E_{K}\otimes_{B(H)}H\otimes G$ is a subspace. Then
 \begin{enumerate}
 \item[(1)] $\cS$  is invariant under $M_S \otimes I_H\otimes I_G$ if and only if there is a Hilbert space $\cD$ and a partial isometry $\Pi:\cF(\cR)\otimes \cD \rightarrow E_{K}\otimes_{B(H)}H\otimes G$ such that, for every $i$,
     $$(M_{S_i}\otimes I_{H\otimes G})\Pi=\Pi (W_i\otimes I_{\cD}) $$ and
     $$\cS=\Pi(\cF(\cR) \otimes \cD).$$
 \item[(2)] Write $\cD_0=H\otimes \cD$  and let $B(H)$ act on $\cD$ in the obvious way. Then $\cS$  is invariant under $M_S \otimes I_H\otimes I_G$ if and only if there is a partial isometry $Y: E_{K^R_c}\otimes_{B(H)}\cD_0 \rightarrow E_{K}\otimes_{B(H)}H\otimes G$ such that, for every $i$,
     $$(M_{S_i}\otimes I_{H\otimes G})Y=Y (M_{S_i}\otimes I_{\cD_0}) $$ and
     $$\cS=Y(E_{K^R_c}\otimes_{B(H)}\cD_0 ).$$
     \end{enumerate}
\end{theorem}
\begin{proof}
In both (1) and (2) the condition is clearly sufficient for $\cS$ to be an invariant subspace. So we attend to the other direction.

For (1), since we assume that $M_S \otimes_{B(H)}I_H \in \overline{D}_{cp}(X,E_K\otimes_{B(H)}H)$ (and, therefore also $M_S \otimes_{B(H)}I_{H\otimes G} \in \overline{D}_{cp}(X,E_K\otimes_{B(H)}H\otimes G)$), we can use Theorem~\ref{Tinvariant}.

For (2) Write $\cD_0:=H\otimes \cD$ (for $\cD$ of part (1)). Fix a unit vector $h_0\in H$ and define $\Pi_0:\cF(\cR)\otimes \cD_0=\cF(\cR)\otimes H\otimes \cD \rightarrow E_K\otimes H\otimes G$ by
$$\Pi_0(\xi \otimes h \otimes d)=\langle h,h_0 \rangle \Pi(\xi \otimes d) .$$
Clearly $\Pi_0$ is a well defined partial isometry with $\cS=\Pi_0(\cF(\cR)\otimes \cD_0)$ satisfying
$$(M_{S_i}\otimes I_{H\otimes G})\Pi_0=\Pi_0 (W_i\otimes I_{\cD_0}). $$
Now, we set $Y=\Pi_0 (\Lambda \otimes I_{\cD})$. Then $(M_{S_i}\otimes I_{H\otimes G})Y=(M_{S_i}\otimes I_{H\otimes G})\Pi_0(\Lambda \otimes I_{\cD})=\Pi_0 (W_i\otimes I_{\cD_0})(\Lambda \otimes I_{\cD})=\Pi_0 (\Lambda \otimes I_{\cD})(M_{S_i}\otimes I_{\cD_0})=Y(M_{S_i}\otimes I_{\cD_0})$.

\end{proof}

\begin{definition}
Suppose $K_1,K_2$ are $\bC$-valued positive kernels on a set $\Sigma$. Then the elements of the reproducing kernel Hilbert space $H_{K_i}$ are functions on $\Sigma$ and, if $G_1,G_2$ are Hilbert spaces we can view the elements of $H_{K_i}\otimes G_i$ as functions from $\Sigma$ to $G_i$ in a natural way.
An operator-valued map $\Theta: \Sigma \rightarrow B(G_1,G_2)$ is said to be a $K_1-K_2$-multiplier if, for every $f\in H_{K_1}\otimes G_1$ (viewed as a function from $\Sigma$ to $G_1$), the function $\Theta f$ lies in $H_{K_2}\otimes G_2$. In such a case, we write $M_{\Theta}$ for the map that sends $f$ to $\Theta f$ and write $\Theta\in \mathcal{M}(H_{K_1}\otimes G_1,H_{K_2}\otimes G_2)$.

A multiplier $\Theta$ is said to be partially isometric if $M_{\Theta}$ is.

\end{definition}

Now note what we get for the case $H=\bC$. In this case $D(X,\bC)=D_c(X,\bC)=\{z=(z_1,\cdots,z_d)\in \bC^d :\; \Sigma_{|\alpha|=|\beta|}z_{\alpha}\overline{z_{\beta}} x_{\alpha,\beta}<1\;\}$.
The kernel $K^R_c$ can be viewed as a $\bC$-valued kernel and, as we saw in Proposition~\ref{Kzw}, $E_{K^R_c}$ is the weighted Hilbert space $\cF(\cR)=\Sigma R_k(\bC^d)^{\circledS k}$. In this case, $K$ in the statement of the theorem is a $\bC$-valued kernel defined on $D_c(X,\bC)\times D_c(X,\bC)$ and giving rise to a RKHS $E_K$ such that \be\label{Xcontractive}
\Sigma_{k=1}^{\infty} M_z^{(k)}(X_k\otimes I_{E_K})M_z^{(k) *}=\Sigma_{|\alpha|=|\beta|}x_{\alpha,\beta} M_{z_{\alpha}}M^*_{z_{\beta}} \leq I.
\ee
Thus, we get
\begin{theorem}\label{BLHC}
Suppose $K:D(X,\bC) \times D(X,\bC) \rightarrow \bC$ is a positive definite kernel where $D(X,\bC) =\{z=(z_1,\cdots,z_d)\in \bC^d :\; \Sigma_{|\alpha|=|\beta|}z_{\alpha} \overline{z_{\beta}} x_{\alpha,\beta}<1\;\}$. Also assume that $M_z$ is a bounded multiplier on the RKHS $E_K$ and it satisfies (\ref{Xcontractive}). Then, given a Hilbert space $G$
 and a subspace $\cS\subseteq E_{K}\otimes G$,
 \begin{enumerate}
 \item[(1)] $\cS$  is invariant under $M_z \otimes I_G$ if and only if there is a Hilbert space $\cD$ and a partially isometric map $\Pi:\cF(\cR) \otimes \cD \rightarrow E_K \otimes G$ such that for every $i$,
     $$(M_{z_i}\otimes I_{ G})\Pi=\Pi (W_i\otimes I_{\cD}) $$ and $$\cS=\Pi(\cF(\cR)\otimes \cD).$$
 \item[(2)] $\cS$  is invariant under $M_z \otimes I_G$ if and only if there is a partially isometric multiplier $\Theta \in \cM(E_{K^R_c}\otimes \cD, E_K\otimes G)$ such that $$\cS=M_{\Theta}(E_{K^R_c}\otimes \cD).$$
 \item[(3)] When $K=K_c^R$ (viewed as a $\bC$-valued kernel), $\cS\subseteq \cF(\cR)\otimes G$  is invariant under $M_z \otimes I_G$ if and only if there is a partially isometric multiplier $\Theta \in \cM(\cF(\cR)\otimes \cD, \cF(\cR)\otimes G)$ such that $$\cS=M_{\Theta}(\cF(\cR)\otimes \cD).$$
     \end{enumerate}
\end{theorem}
\begin{proof}
Everything follows from Theorem~\ref{BLH} except for the fact that the map $Y$ in part (2) is $M_{\Theta}$ for some multiplier $\Theta \in \cM(E_{K^R_c}\otimes \cD, E_K\otimes G)$. Applying Theorem~\ref{BLH}(2), we get a partial isometry $Y: E_{K^R_c}\otimes\cD \rightarrow E_{K}\otimes G$ such that, for every $i$,
     $$(M_{z_i}\otimes I_{ G})Y=Y (M_{z_i}\otimes I_{\cD}) $$ and
     $$\cS=Y(E_{K^R_c}\otimes \cD ).$$
To prove that $Y=M_{\Theta}$ we first follow the proof of \cite[Lemma 2.2]{S2015} (using also Lemma~\ref{Mzi}) to show that for every $w\in D(X,\bC)$ there is a map $\Theta(w) \in B(\cD,G)$ such that, for every $g\in G$,
\be
Y^*(k_w \otimes g)=k^R_w \otimes \Theta(w)^*g
\ee
where $k_w, k^R_w$ are the reproducing functions of $K, K^R_c$ respectively.
Write $A(z,w)=K(z,w)I_G-\Theta(z)K^R_c(z,w)\Theta(w)^*\in B(G)$ and compute,
$$A(z,w)=\langle k_z,k_w\rangle I_G-\Theta(z)\langle k^R_z,k^R_w\rangle \Theta(w)^*.$$ Thus,
 for $g,h\in G$,
$$\langle A(z,w)g,h\rangle=\langle k_w,k_z\rangle \langle g,h \rangle-\langle k^R_w \otimes \Theta(w)^*g,k^R_z \otimes \Theta(z)^*h\rangle  =$$  $$ \langle k_w,k_z\rangle \langle g,h \rangle-\langle Y^*(k_w\otimes g),Y^*(k_z\otimes h)\rangle =\langle k_w,k_z\rangle \langle g,h \rangle-\langle k_w\otimes g,YY^*(k_z\otimes h)\rangle=$$ $$\langle k_w\otimes g,(I-YY^*)(k_z\otimes h)\rangle.$$ It follows that $A(z,w)$ is positive definite and, therefore, $\Theta$ is a multiplier.
\end{proof}

Note that part (3) of the theorem is closely related to \cite[Theorem 3.3]{Po2010} and to \cite[Theorem 2.3]{S2015}.

Restricting to the case $R_k=I $ for all $k$ (or, equivalently, $X_1=I$ and $X_k=0$ for $k>1$) we get \cite[Theorem 6.5]{BEKS}.

In the analysis above, condition (\ref{Xcontractivecondition}) plays an important role. It generalizes the condition that $M_S$ is contractive. In order to better understand condition (\ref{Xcontractivecondition}) we shall first need the following lemma.

\begin{lemma}\label{PhiTL}
For $T,L$ in $D_c(X,H)$ the map $\Phi_{T,L}:B(H)\rightarrow B(H)$ defined by
$$\Phi_{T,L}(a)=\Sigma_{k=1}^{\infty} T^{(k)}(X_k \otimes a)L^{(k)} $$ is bounded with $||\Phi_{T,L}||<1$ and
$$(id-\Phi_{T,L})^{-1}=K^R_c(T,L) .$$
This can be viewed as the operator version of (\ref{xr}).
\end{lemma}
\begin{proof}
Recall (Equation (\ref{PhiT})) that $\Phi_T(a)=\Sigma_{k=1}^{\infty}T^{(k)}(X_k\otimes a)T^{(k)*}$. This is a completely positive map on $B(H)$ and, thus, $||\Phi_T||=||\Phi_T(I)||$.

Fix $a\in B(H)$ and consider the rows of operators $A:=(T^{(k)}(X_k^{1/2}\otimes a))_{k=1}^{\infty}$ and $B:=(L^{(k)}(X_k^{1/2}\otimes I_H))_{k=1}^{\infty}$. Then $AA^*=\Phi_T(aa^*)$, $BB^*=\Phi_L(I)$ and $AB^*=\Phi_{T,L}(a)$. Thus
$$||\Phi_{T,L}(a)||\leq ||A||||B||\leq ||\Phi_T||^{1/2}||a||||\Phi_L||^{1/2}.$$
Since $T$ and $L$ lie in $D_c(X,H)$, $||\Phi_T||<1$ and also $||\Phi_L||<1$ and it follows that $||\Phi_{T,L}||<1$. Therefore $(id-\Phi_{T,L})^{-1}$ is a well defined map on $B(H)$ that is equal to $\Sigma_{k\geq 0}\Phi_{T,L}^k$.
A computation very much like in \cite[page 516]{MSWS} shows that
$$(id-\Phi_{T,L})^{-1}=K^R_c(T,L) .$$
(There it is done for $T=L$).
\end{proof}

\begin{theorem}\label{condK}
For a kernel $K:D_c(X,H)\times D_c(X,H)\rightarrow B_*(B(H),B(H))$ and $\Phi_{T,L}$ as in Lemma~\ref{PhiTL},
\begin{enumerate}
\item[(1)] Condition (\ref{Xcontractivecondition}) is equivalent to the condition that
$(id - \Phi_{T,L})\circ K(T,L) $ is a cp kernel.
\item[(2)]  Condition (\ref{Xcontractivecondition}) is equivalent to the condition that
$(K^R_c(T,L))^{-1}\circ K(T,L) $ is a cp kernel.
\end{enumerate}
\end{theorem}
\begin{proof}
Once we prove part (1), part (2) will follow from Lemma~\ref{PhiTL}. So it suffices to prove part (1).

$(id - \Phi_{T,L})\circ K(T,L) $ is a cp kernel means that, for every $T^1,\ldots,T^k,L^1,\ldots,L^k$ in $D_c(X,H)$ and every $a_1,\ldots,a_k$ in $B(H)$, the matrix (with entries in $B(H)$) defined by
$$((id - \Phi_{T^i,L^j})\circ K(T^i,L^j)(a_ia_j^*))_{i,j=1}^k $$
is positive.
Since
\be
 K(T^i,L^j)(a_ia_j^*)=\langle k_{a_i,T^i},k_{a_j,L^j} \rangle
\ee
we have
$$\Phi_{T^i,L^j}( K(T^i,L^j)(a_ia_j^*))=\Sigma_{|\alpha|=|\beta|} x_{\alpha,\beta}T^i_{\alpha}\langle k_{a_i,T^i},k_{a_j,L^j} \rangle L^{j *}_{\beta}=$$  $$\Sigma x_{\alpha,\beta}\langle k_{a_i,T^i}\cdot T^{i *}_{\alpha},k_{a_j,L^j}\cdot L^{j *}_{\beta} \rangle = \Sigma x_{\alpha,\beta}\langle M_{S_{\alpha}}^*k_{a_i,T^i},M_{S_{\beta}}^*k_{a_j,L^j}\rangle=$$  $$\langle k_{a_i,T^i},\Sigma x_{\alpha,\beta}M_{S_{\alpha}}M_{S_{\beta}}^*k_{a_j,L^j} \rangle $$ and $$(id - \Phi_{T^i,L^j})\circ K(T^i,L^j)(a_ia_j^*)=\langle k_{a_i,T^i},(I-\Sigma x_{\alpha,\beta}M_{S_{\alpha}}M_{S_{\beta}}^*)k_{a_j,L^j}\rangle.$$
Thus, the positivity of $(id - \Phi_{T,L})\circ K(T,L) $  is equivalent to $I-\Sigma x_{\alpha,\beta}M_{S_{\alpha}}M_{S_{\beta}}^*\geq 0$ which is (\ref{Xcontractivecondition}).

\end{proof}

\begin{example}
Suppose $\{B^2_k\}_{k=0}^{\infty}$ is a sequence of positive operators, $B_k\in B((\bC^d)^{\otimes k})$, and the kernel $K^B(T,L)(a):=\Sigma_{k=0}^{\infty} T^{(k)}(B_k^2\otimes a)L^{(k) *}$ is well defined on $D_c(X,H)$ (where $a\in B(H)$). Define $$C_k^2=\Sigma_{m=0}^{k} R_m^2 \otimes B_{k-m}^2 \in B((\bC^d)^{\otimes k})$$ and write
$$K^C(T,L)(a):=\Sigma_{k=0}^{\infty} T^{(k)}(C_k^2\otimes a)L^{(k) *}.$$
Then a straightforward computation shows that
$$K^C(T,L)=K^R_c(T,L)\circ K^B(T,L) .$$ Thus, using Theorem~\ref{condK} (2), $K^C$ satisfies condition (\ref{Xcontractivecondition}) and $E_K$ is an $(X,H)$-contractive reproducing kernel correspondence.

In particular, this holds when $B_k=R_k$ so that
$$C_k^2=\Sigma_{m=0}^{\infty} R_m^2 \otimes R^2_{k-m} .$$ For example, if $R_k=I$ for all $k$, we get $C_k^2=(k+1)I$ and $E_{K^C}$ can be viewed as a generalization of the Bergman space.

\end{example}


\end{document}